\documentclass[11pt,reqno]{amsart}
\usepackage{mathptmx}      % use Times fonts if available on your TeX system
\usepackage{amssymb}
\usepackage{amscd}
\usepackage{amsmath}
\usepackage{amsthm}
\usepackage{amsrefs}
\usepackage{enumitem}
\usepackage{scrtime}

\newcommand{\supp}{{\rm supp \;}}
\newcommand{\sign}{{\rm \, sign \;}}
\newcommand{\R}{\mathbb{R}}
\newcommand{\N}{\mathbb{N}}
\newcommand{\Z}{\mathbb{Z}}
\newcommand{\T}{\mathbb{T}}

\newcommand{\meas}{\Omega}
\newcommand{\BM}[3]{\mathcal{M}_{#1}^{#2}\left(2,#3\right)}
\newcommand{\wBM}[3]{\mathcal{M}_{#1}^{#2,\infty}\left(2,#3\right)}

\newcommand{\mul}{{\bf m}}

\newcommand{\dd}{\mathrm{d}}

\newcommand{\norm}[1]{\left\Vert#1\right\Vert}
\newcommand{\brkt}[1]{\left(#1\right)}
\newcommand{\abs}[1]{\left|#1\right|}
\newcommand{\set}[1]{\left\{#1\right\}}
\newcommand{\esc}[1]{\langle{#1}\rangle}

    \newtheorem{thm}{Theorem}[section]
    \newtheorem{cor}[thm]{Corollary}
    \newtheorem{lem}[thm]{Lemma}
    \newtheorem{prop}[thm]{Proposition}
    \newtheorem{defn}[thm]{Definition}
    \newtheorem{rem}[thm]{Observation}

    \numberwithin{equation}{section}

\begin{document}
%\date{\today, \thistime}
\date{August 21, 2012}
\title{A Homomorphism Theorem for Bilinear Multipliers}
\author[S.~Rodr\'iguez-L\'opez]{Salvador Rodr\'iguez-L\'opez}
\address{Department of Mathematics, Uppsala University, Uppsala, SE 75106,  Sweden}
\email{salvador@math.uu.se}
\urladdr{http://www.math.uu.se/~salvador}
\thanks{The author has been partially supported by the Grant MTM2010-14946.}

\maketitle

\begin{abstract}
In this paper we prove an abstract homomorphism theorem for bilinear multipliers in the setting of locally compact Abelian (LCA) groups. We also provide some applications. In particular, we obtain  a bilinear abstract version of K. de Leeuw's theorem for bilinear multipliers of strong  and weak type. 
We also obtain necessary conditions on bilinear multipliers on non-compact LCA groups, yielding boundedness for the corresponding operators on products of rearrangement invariant spaces.  Our investigations extend some existing results in $\R^n$ to the framework of general LCA groups, and  yield new boundedness results for bilinear multipliers in quasi Banach spaces. 
\end{abstract}

\section{Introduction}
The study of multilinear multipliers is motivated by their natural appearance in analysis, such as in the work of R. Coifman and Y. Meyer on singular integral operators and commutators \cite{MR518170}.  The proof of M. Lacey and C. Thiele (see \cites{MR1689336}) on the boundedness of the bilinear Hilbert transform, ignited interest in questions related to multilinear operators, which lead to the study of the validity of multilinear counterparts to classical linear results.   In particular, and of direct relevance to this paper, there has been quite a few studies in establishing multilinear versions of K. de Leeuw's type theorems (see \cite{MR0174937}) on the Lebesgue spaces \cites{MR1808390, MR2301463, MR2471164, MR2169476, MR2037006}. The proofs in the existing literature, rely either on the dilation structure of $\R^n$ or on duality arguments that use the Banach space structure of the target space. 

Roughly speaking, de Leeuw's results state that if $\mul$ is a Fourier multiplier for $L^p(\R^n)$, with $1\leq p\leq \infty$, then if $\pi$ is either the natural injection of $\Z^n$ in $\R^n$ or that of $\R^{d}$ in $\R^n$  for $d<n$, the composition $\mul\circ \pi$ is also a multiplier for $L^p(\T^n)$, respectively for $L^p(\R^d)$, with norm bounded by the norm of $\mul$.  These results were generalised to the context of LCA groups first by S. Saeki \cite{MR0275057}, and later reproved, using transference techniques, by R. Coifman and G. Weiss \cite{MR0481928}. Applying these transference ideas, N. Asmar \cite{MR1001119} and E. Berkson, T.A. Gillespie and P. Muhly \cite{MR1004717} obtained a proof of R. Edwards and G. Gaudry's homomorphism theorem for multipliers \cite[Theorem B.2.1]{MR0618663}, which allows to recover de Leeuw's result as a special case. 

The aim of this work is to obtain,  in the abstract setting of LCA groups, a homomorphism theorem for bilinear multipliers (see Theorem \ref{thm:main theorem} below), which is the bilinear counterpart of  Edwards and Gaudry's. Roughly speaking, we show that if $G$ and $\Gamma$ are two LCA groups, $\mul$ is a bilinear multiplier on $G$ and $\pi$ is a  homomorphism between the dual group of $\Gamma$ and $G$, then the composition $\mul \circ \pi\otimes \pi$ is also a bilinear multiplier on $\Gamma$, with operator norm bounded by the norm of $\mul$.

In contrast to the linear case, interesting  multilinear operators, such as the bilinear Hilbert transform,  or bilinear Calder\'on-Zygmund operators, map Banach Lebesgue spaces to $L^p$ spaces with $0 < p < 1$.  Thus duality is precluded in proving the most general results.  

The two main difficulties to develop the abstract theory are the lack of duality for target spaces and of dilation structure in the general setting. The main achievements of this work are to provide proofs that rely only on the underlying group structure (avoiding dilation arguments), using the bilinear transference techniques developed by L. Grafakos and G. Weiss \cite{MR1398100} (see  also O. Blasco, M. Carro and T. A. Gillespie's work \cite{MR2169476} for a related approach),  and moreover, to develop  a method of approximating bilinear Fourier multipliers between general rearrangement invariant function spaces, in particular Lebesgue spaces, to tackle the technical difficulties of dealing with non-Banach target spaces (Theorem \ref{thm:Main} below). 

As application of our study, we recover several known results and present some new ones. In particular, we obtain an abstract de Leeuw's type theorem (Theorem \ref{thm:DeLeeuw}) that allows us to extend D. Fan and S. Sato's results (see Corollary \ref{cor:Fan_Sato} below) for anisotropic dilations, and to extend G. Diestel and L. Grafakos's \cite{MR2301463}*{Proposition 2} for $p<1$ and for weak type multipliers.  Furthermore, inspired by N. Asmar and E. Hewitt's approach  in the linear setting \cite{MR930884}, we define a Generalised Bilinear Hilbert Transform on certain groups with ordered dual, and obtain an abstract version of Lacey and Thiele's result for it (see Theorem \ref{thm:generalisedBHT} below).  As another application we obtain necessary conditions, in terms of the Boyd indices, on multipliers on non-compact LCA groups to be bounded on products of rearrangement invariant spaces (see Theorem \ref{thm:necessity} below). This is a bilinear counterpart of  the classical result of L. H\"ormander \cite[Theorem 1.1]{MR0121655}. In particular, our result extends L. Grafakos and R. Torres's \cite[Proposition 5]{MR1880324}, L. Grafakos and J. Soria's \cite[Proposition 2.1]{MR2595656} and F. Villarroya's \cite[Proposition 3.1 ]{MR2471164}, to the setting of multipliers on general non-compact LCA groups acting on rearrangement invariant spaces.

The paper is organised as follows: In Section \ref{sect:notation} we introduce the basic notations and state our main results, which we prove in sections \ref{sect:proof_main} and \ref{sect:approx} respectively. 
Applications derived from our main theorems are collected in Section \ref{sect:Application and consequences}. 

It is worth mentioning that the results of this work easily extends to the setting of $m$-linear operators when $m\geq 3$ but, for the sake of simplicity in the exposition, we restrict our discussion  to the bilinear case as it contains the major ideas of this investigation. 

\section{Basic notation and main results}\label{sect:notation}
Here $G$ denotes a locally compact Hausdorff,
$\sigma$-compact, Abelian topological group and we shall abbreviate it to LCA group. We adopt the additive notation for the group inner operation.  We shall denote by $\widehat{G}$ the group of characters and we write $\esc{\xi,x}$ for the value of $\xi\in \widehat{G}$ at $x\in G$, and  $\overline{\esc{\xi,x}}$ for its complex conjugate. We shall use the letters $x,y$ for denoting elements in $G$, and $\xi,\eta,\zeta,\gamma$ for elements in $\widehat{G}$. We reserve the symbol $e_G$ for the identity element of $G$. In order to avoid technical conditions, we will assume that the group $G$ is metrisable which is equivalent for $\widehat{G}$ to be $\sigma$-compact. 

From now on, $L^{1}(G)$ stands for the space of integrable functions on $G$ with respect the Haar measure, and we denote by $L^1_c(G)$ the subspace of compactly supported integrable functions. Let  $\widehat{f}$ be the Fourier transform of a function $f$ defined by
\[
    \widehat{f}(\xi)=\int_G f(u) \overline{\esc{\xi,u}}\, {\dd u}.
\]
We choose the Haar measure in $\widehat{G}$ in such a way that the following Fourier inversion formula holds,
\[
	f(u)=\int_{\widehat{G}} \widehat{f}(\xi) \esc{\xi,u}\, \dd \xi,
\]
for any $f\in SL^1(G)$, which stands for the space of function $f\in L^1(G)$ such that $\widehat{f}\in L^1(\widehat{G})$. We shall denote by $f^\vee$ the inverse Fourier transform defined by $f^\vee(\xi)=\widehat{f}(-\xi)$.  We write $G^2$ for denoting the group $G\times G$ endowed with the product measure. For any functions $f,g$ on $G$ we introduce another function $f\otimes g$ on $G^2$ by setting
$
    f\otimes g(\xi,\eta)=f(\xi) g(\eta).
$

For more information about topological groups and their properties we refer the reader to \cite{MR551496}.

By a quasi-Banach function space (QBFS for short) on a totally $\sigma$-finite measure space $\brkt{\meas,\Sigma, \mu}$,
we denote a complete linear subspace $X$ of the space of $\mu$-measurable functions, $L^0(\meas)$, endowed
with a (quasi-)norm $\norm{\cdot}_X$ with the following properties:
\begin{enumerate}
    \item $f\in X$ if, and only if $\norm{f}_X=\norm{\,\abs f\,}_X<\infty$;
    \item $g \in X$ and $\norm{g}_X \leq \norm{f}_X$, whenever
$g \in L^0(\meas)$, $f \in X$, and $|g| \leq |f |$ $\mu$-a.e;
    \item If $0 \leq f_n\uparrow f$ a.e., then $\norm{f_n}_X \uparrow \norm{f}_X$;     \item $\mu(E)<\infty \Rightarrow ||\chi_E||_{X}<\infty$.
\end{enumerate}
Observe that bounded functions supported in sets of finite measure
belong to every QBFS. If $\norm{\cdot}_X$ is a norm, and for any
finite measure set $E$, there exists a constant $C_E$ such that,
\begin{equation}\label{eq:technic_20}
    \int_E \abs f\leq C_E \norm{f}_X,
\end{equation}
we say that $X$ is a Banach function space (BFS for short). The following {\it Fatou's property} holds:
\begin{lem}\label{lem:Fatou}\cite[Lemma I.1.5]{MR928802} Let $X$ be a QBFS, and, for $n\in \N$, $f_n\in X$.
 If $f_n\to f$ a.e., and if $\liminf_n
    \norm{f_n}_X<\infty$, then $f\in X$ and
    \[
        \norm{f}_X\leq \liminf _n \norm{f_n}_X.
    \]
\end{lem}

We say that a QBFS (or a BFS) $X$ is rearrangement invariant (RI for
short) if there exists a quasi-norm (respectively a norm)
$\norm{\cdot}_{X^*}$ defined on $L^0\left[0,\mu(\meas)\right)$
endowed with the Lebesgue measure, such that
$\norm{f}_X=\norm{f^*}_{X^*}$. Here $f^*$ stands for the
non-increasing rearrangement of $f$, defined, for $t> 0$, by
\[
       f^*(t)=\inf\set{s:\; \mu_f(s)\leq t},
\]
where $\mu_f(s)=\mu\set{x:\;|f(x)|>s}$ is the distribution function
of $f$. Let  $X$ be a RI QBFS, let
\begin{equation}\label{Dilation}
	E_{1/s}f^*(t)=f^*\brkt{t/s}, \quad s,t>0,
\end{equation} be the dilation operator,
and denote  by $h_X(s)$ its norm. That is,
\begin{equation}\label{Dilation_norm}
h_X(s)=\sup_{f\in X\setminus \set{0}}
    \frac{\norm{E_{\frac{1}{s}}f^*}_{X^*}}{\norm{f^*}_{X^*}},\,
    s>0.
\end{equation} 
Lebesgue $L^p$ spaces, Classical Lorentz spaces and Orlicz spaces are examples of RI QBFSs. We refer the reader to \cite{MR928802} for further information on non-increasing rearrangement, BFS and RI spaces.

It is is easy to see for that any BFS $X$ on $G$ equipped with the Haar measure, such that $\norm{\cdot}_X$ is absolutely continuous (see \cite[Definition 3.1, p. 14]{MR928802}), $SL^1(G)\cap X$ is a dense set in $X$. In particular, $SL^1(G)$ is dense in any $L^p(G)$ for $p<\infty$.  

A QBFS $X$ is the $p$-convexification of a BFS $Y$ if $X$ can be renormed by a quasi-norm $\norm{\cdot}_X$ such that, for any $f\in X$,
\[
	\norm{f}_{X}=\norm{\abs{f}^p}_Y^{1/p}.
\]
In such case, we will assume that the quasi-norm in $X$ is given by $\norm{\abs{.}^p}_Y^{1/p}$. The Lebesgue $L^q$ and the weak Lorentz $L^{q,\infty}$ spaces, for $0<q<1$ and $L^{1,\infty}$ are examples of such spaces. 

A BFS $X$ is $p$-concave (see \cite{MR540367}) if there exists a constant $M<\infty$ so that 
\[
	\brkt{\sum_{j=1}^n \norm{f_j}_X ^p}^{\frac 1 p}\leq M \norm{\brkt{\sum_{j=1}^n \abs{f_j}^p}^{\frac 1 p}}_X,
\]
for every choice $\{f_j\}_{j=1}^n$ in $X$. The least constant $M$ satisfying the inequality is denoted by $M_{(p)}(X) $.  Let us observe that, for any $1\leq p<\infty$, $M_{(p)}(L^p)=1$. 

Throughout the paper, we shall assume that $X_1,X_2$ are RIBFS and $X$ is a RI QBFS on $G$ endowed with the Haar measure. 

\begin{defn}
Let $\mul(\xi, \eta)\in L^\infty(\widehat{G}^2)$. Define
\[
 B_{\mul}(f, g)(x)=\iint_{\widehat{G}^2}\hat{f}(\xi)\hat{g}(\eta)\mul(\xi, \eta)\langle\xi+\eta,x\rangle\, {\dd \xi} {\dd \eta}
\]
for $f,\ g\in SL^1(G)$. We say that $\mul$ is a bilinear multiplier for $(X_1,X_2,X)$ if there exists $C>0$
such that
\begin{equation}\label{eq:eq_1}
||B_{\mul}(f, g)||_{X}\leq C||f||_{X_1}||g||_{X_2}
\end{equation}
for any $f,\ g\in SL^1(G)$.  We write $\BM{X_{1},X_{2}}{X}{G}$ for the space of bilinear multipliers for $(X_1,X_2,X)$, and we denote by $\norm{\mul}_{\BM{X_{1},X_{2}}{X}{G}}$ the least constant $C$ satisfying \eqref{eq:eq_1}.
\end{defn}

If $(X_1,X_2)=(L^{p_1},L^{p_2})$ and either $X= L^{p}$ or $X=L^{p,\infty}$, we will write it simply $\BM{p_1,p_2}{p}{G}$, $\wBM{p_1,p_2}{p}{G}$ respectively, for short.

Observe that if $f,g\in SL^1(G)$, then $f\otimes g\in SL^1(G^2)$ and $\mul(\widehat{f}\otimes \widehat g)\in L^1(\widehat{G}^2)$. Then,
$B_\mul(f,g)(x)$ makes pointwise meaning as a continuous function. Observe also that if $\mul=\widehat{K}$ where $K\in L^1_{\rm c} (G^2)$, for any $f,g\in SL^1(G)$, 
\[
    B_{\mul}(f,g)(x)=\iint_{G^2} K(u,v) f(x-u)g(x-v)\, {\dd u}\, {\dd v}.
\]

Here and subsequently, $\mathfrak{c}$ stands for a universal constant that depends only on $0<p,p_1,p_2<\infty$, which value is given by
\[
   \mathfrak{c}=\left\{
                    \begin{array}{ll}
                      \frac{B_{p_1} B_{p_2}}{A_p^2}, & \hbox{if $p<1$;} \\
                      1, & \hbox{if $p\geq 1$},
                    \end{array}
                  \right.
    {},
\]
where, $A_q$ and $B_q$ denotes the best constant on Khintchine's
inequality (see \cite{MR540367}*{Theorem 2.b.3})
$$
    A_q\brkt{\sum \abs{\alpha_j}^2}^{1/2}\leq \norm{\sum_j \alpha_j r_j}_{L^q[0,1]}\leq  B_q \brkt{\sum \abs{\alpha_j}^2}^{1/2}.
$$
Here $\{r_j\}$ stands for the Rademacher's system. 

Our main results can be stated as follows.
\begin{thm}[Homomorphism theorem for bilinear multipliers]\label{thm:main theorem}
Let $G,\Gamma$ be LCA groups and let $\pi:\widehat{G}\to \widehat{\Gamma}$ be
a group homomorphism. Let $\mul\in \mathcal{C}_b(\widehat{\Gamma})$.  Suppose that $1\leq p_{1},\  p_{2}<\infty$, $0<p\leq\infty$ satisfy
\begin{equation*} 
    \frac{1}{p_1}+\frac{1}{p_2}=\frac{1}{p}.
\end{equation*}
The following holds:
\begin{enumerate}
	\item If $\mul\in \BM{p_1,p_2}{p}{\Gamma}\cap \mathcal{C}_b(\widehat \Gamma)$, then $\mul\circ (\pi\otimes \pi)\in
\BM{p_1,p_2}{p}{G}$ and
\[
    \norm{\mul\circ (\pi\otimes \pi)}_{\BM{p_1,p_2}{p}{G}}\leq \mathfrak{c} \norm{\mul}_{\wBM{p_1,p_2}{p}{\Gamma}}.
\]
	\item If $\mul\in  \wBM{p_1,p_2}{p}{\Gamma}\cap \mathcal{C}_b(\widehat \Gamma)$, then $\mul\circ (\pi\otimes \pi)\in
\wBM{p_1,p_2}{p}{G}$  and
\[
    \norm{\mul\circ (\pi\otimes \pi)}_{\wBM{p_1,p_2}{p}{G}}\leq \mathfrak{c} \norm{\mul}_{\wBM{p_1,p_2}{p}{\Gamma}}.
\]
\end{enumerate}
\end{thm}

\begin{rem}\label{rem: normalized} The condition $\mul\in \mathcal{C}_b(\widehat{G})$ can be relaxed  to the assumption of $\mul$ being normalized (see Definition \ref{def:normalized} below). Indeed, the result holds if $\mul$ is continuous on the image of $\widehat{G}^2$ by $\pi\otimes \pi$ (see Remark \ref{eq:final_remark} below).
\end{rem}
The proof of the previous results rely on a general approximation property of bilinear multipliers, that is the bilinear analogue of \cite{MR2609318}*{Lemma 2}.

\begin{thm}\label{thm:Main} Let $0<p\leq 1\leq p_1,p_2<\infty$. Let $X$ be the $p$-convexification of a RIBFS and let $X_1,X_2$ be RIBFS,  such that $X_i$ is $p_i$-concave for $i=1,2$. For any $\mul\in L^\infty(\widehat{G}^2)\cap\BM{X_{1},X_{2}}{X}{G}$ there
exists a sequence $\{\mul_j\}_j\subset L^\infty(\widehat{G}^2)$ such that:
\begin{enumerate}
    \item[($P_1$)] for each $j$, $\mul_j^{\vee}\in L^1_c(\widehat{G}^2)$;
    \item[($P_2$)]  \label{condition} for almost every $\xi,\eta\in \widehat{G}$, 
    $\lim_j \mul_j(\xi,\eta)=\mul(\xi,\eta)$;
    \item[($P_3$)]  $\sup_j \Vert\mul_j\Vert_\infty\leq \norm{\mul}_\infty$,
    \item[($P_4$)]  $\sup_j \Vert{\mul_j}\Vert_{\BM{X_{1},X_{2}}{X}{G}}\leq \mathfrak{d} 
    \norm{\mul}_{\BM{X_{1},X_{2}}{X}{G}}$,
\end{enumerate} 
where $\mathfrak{d}=1$ if $p\geq 1$, or $\mathfrak{d}=M_{(p_1)}(X_1)M_{(p_2)}(X_2) \mathfrak{c}$ otherwise.
Moreover, if $\mul\in \mathcal{C}_b(\widehat{G})$ or $\mul$ is normalized (see Definition \ref{def:normalized} below) then
\begin{enumerate}
	\item[($P_2^\prime$)]  \label{condition2} for every $\xi,\eta\in \widehat{G}$, 
    $\lim_j \mul_j(\xi,\eta)=\mul(\xi,\eta)$.
\end{enumerate}
\end{thm}

\section{Proof of Theorem \ref{thm:Main}}\label{sect:proof_main}

 In order to prove the Theorem we need first to prove some technical lemmas. Let us denote by $M_\xi f(x)=\overline{\esc{\xi,x}} f(x)$.

\begin{lem}\label{lem:technic_1} 
    Let $f,g\in SL^1(G)$. For any $x\in G$, the function
    \[
        \widehat{G}^2\ni(\zeta,\gamma)\mapsto F_x(\zeta,\gamma):=B_{\mul} (M_{-\zeta} f, M_{-\gamma} g)(x),
    \]
    is uniformly continuous (uniformly on $x$). Moreover, for any
    $n$, there exists a symmetric relatively compact open neighbourhood $U_n$ of
    $e_{\widehat{G}}$, such that $U_{n}+U_n\subset U_{n-1}$ and,  for any $\zeta,\zeta^\prime\in U_n$
    \[
        \sup_{x\in G, \gamma\in \widehat{G}} \abs{F_x(\zeta,\gamma)-F_x(\zeta^\prime,\gamma)}\leq \frac{1}{n}.
    \]
\end{lem}
\begin{proof} Let $\zeta,\zeta^\prime, \gamma\in \widehat{G}$, $x\in G$. Since Haar measure is invariant under translations, it holds
\begin{eqnarray*}
    \lefteqn{\left|F_{x}(\zeta,\gamma) - F_{x}(\zeta^\prime,\gamma)\right|=}& &\\
    &=& \iint_{\R^2}
    \mul(\xi,\eta)
    \brkt{\widehat{f}(\xi+\zeta)-\widehat{f}(\xi+\zeta^\prime)}
    \widehat{g}(\eta+\gamma) \esc{\xi+\eta, x}\, {\dd \xi}\, {\dd \eta}\\
    &\leq& \norm{\mul}_{\infty}
    \norm{\widehat{f}-\tau_{\zeta-\zeta^\prime}\widehat{f}}_{L^1(\widehat{G})}
    \norm{\widehat{g}}_{L^1(\widehat{G})},
\end{eqnarray*}
where $\tau_\zeta$ stands for the translation operator. 
Then,
\[
    \sup_{x\in G.\gamma\in \widehat{G}} \left|F_{x}(\zeta,\gamma) - F_{x}(\zeta^\prime,\gamma)\right|\leq \norm{\mul}_{\infty}
    \norm{\widehat{f}-L_{\zeta-\zeta^\prime}\widehat{f}}_{L^1(\widehat{G})}
    \norm{\widehat{g}}_{L^1(\widehat{G})}.
\]
The result easily follows by the uniform continuity of translations  in $L^1(\widehat{G})$ \cite{MR551496}*{(20.4)}.
\end{proof}

\begin{lem}[Marcinkiewicz-Zygmund's bilinear inequality] \label{lem:MZ}Let $X_1,X_2$ be BFSs and let $X$ be a QBFS. Assume that for some $0<p\leq 1\leq p_1,p_2<\infty$, $X$ is the $p$-convexification of a BFS and $X_j$ is $p_j$-concave $j=1,2$. 
If $T$ is a bounded bilinear operator such that
\[
    \norm{{T(f,g)}}_{X}\leq \norm{T} \norm{f}_{X_1}\norm{g}_{X_2},
\]
then
\[
    \norm{\brkt{\sum_{j,k} \abs{T(f_j,g_k)}^2}^{1/2}}_{X}\leq  \mathfrak{d}\norm{T} \norm{\brkt{\sum_{j} \abs{g_j}^2}^{1/2}}_{X_1}  \norm{\brkt{\sum_{k} \abs{g_k}^2}^{1/2}}_{X_2},
\]
 for any family $\{f_j\}_j\subset X_1$, $\{g_j\}_j \subset X_2$ and $\mathfrak{d}$ as above.
\begin{proof}
	Observe that is suffices to prove the result for $\{f_j\}$ and $\{g_k\}$ with a finite number of elements. The assumption on $X$ implies that $\norm{f}_{X}=\norm{\abs{f}^p}_Y^{1/p}$ where $Y$ is a BFS. Khintchine's bilinear inequality \cite{MR0290095}*{Appendix
    D}, and the $p$-convexity of the space $X$ yield
   \begin{eqnarray*}
            \lefteqn{\norm{\brkt{\sum_{j,k} \abs{T(f_j,g_k)}^2}^{1/2}}_{X}\leq \frac{1}{A_p^2} \norm{\iint_{[0,1]^2} \abs{\sum_{j,k}r_j(s)r_k(t)
            T(f_j,g_k)}^p \, {\dd s }\, {\dd t }}_Y^{1/p}}\\
            &\leq& \frac{1}{A_{p_3}^2}\brkt{\iint_{[0,1]^2} \norm{{T\brkt{\sum_j r_j(s) f_j, \sum_k r_k(t)g_k }}}_{X}^p
             \, {\dd s }\, {\dd t }}^{1/p}\\
            &\leq& \frac{\norm{T}}{A_{p_3}^2}\brkt{\int_0^1 \norm{\sum_j r_j(s)
            f_j}_{X_1}^p\, {\dd s }}^{1/p}\brkt{\int_0^1 \norm{\sum_k r_k(t)
            g_k}_{X_2}^p\, {\dd t }}^{1/p}.
    \end{eqnarray*}
    Since $X_1$ is $p_1$-concave and $p\leq
    1$, H\"older inequality and \cite{MR540367}*{Theorem 1.d.6} yield
    \[
        \brkt{\int_0^1 \norm{\sum_j r_j(s)
            f_j}_{X_1}^p\!\!\!\!  {\dd s }}^{\frac 1 p}\leq {\int_0^1 \norm{\sum_j r_j(s)
            f_j}_{X_1}\!\!\!\!   {\dd s }}\leq M_{(p_1)}(X_1) B_{p_1} \norm{\brkt{\sum_{j}
            \abs{f_j}^2}^{\frac 1 2}}_{X_1}.
    \]
    This finishes the proof because a similar inequality holds for the other term.
\end{proof}
\end{lem}

The following result extends \cite{MR2172393}*{Lemma 2.2} for the case where the target space is not Banach and it is the bilinear unweighed analogue of \cite{MR2888205}*{Lemma 3.6}. Before we discuss it, we introduce some notation. We denote by $M(\widehat{G})$ the space of complex measures $\lambda$ defined on $\widehat{G}$, with finite total variation $\norm{\lambda}_{M(\widehat{G})}=\int_G \dd \abs{\lambda}(x)$. The convolution of a complex measure and a function is defined in the usual way as in \cite{MR551496}*{(20.12)}. We say that a bounded function $\mul$ is a Fourier multiplier for $X$ if the operator defined on $SL^1(G)$ by 
\[
	T_\mul f(x)=\int_{\widehat{G}} \mul(\xi)\widehat{f}(\xi)\esc{\xi,x}\, \dd \xi,
\]
extends to a Bounded operator on $X$.   We write $\mathcal{M}_X(G)$ for denoting the space of linear multipliers acting on $X$ and $\Vert{\mul}\Vert_{\mathcal{M}_X(G)}$ denotes the norm of the associated operator $T_\mul$. 

\begin{prop}\label{prop:clave}  Let $0<p\leq 1\leq p_1,p_2<\infty$. Let $X$ be the $p$-convexification of a RIBFS and let $X_1,X_2$ be RIBFS,  such that $X_i$ is $p_i$-concave for $i=1,2$. Let $\mul\in \BM{X_{1},X_{2}}{X}{G}$. The following holds:
\begin{enumerate}
    \item If $\mul_{1}\in \mathcal{M}_{X_1}(G)$ and $\mul_{2}\in \mathcal{M}_{X_{2}}(G)$, then
    $(\mul_{1}\otimes \mul_2)\mul\in \BM{X_{1},X_{2}}{X}{G}$ and 
              \[
                \norm{(\mul_{1}\otimes \mul_2)\mul}_{ \BM{X_{1},X_{2}}{X}{G}}\leq \norm{\mul}_{ \BM{X_{1},X_{2}}{X}{G}}\norm{\mul_1}_{\mathcal{M}_{X_1}(G)}\norm{\mul_{2}}_{\mathcal{M}_{X_2}(G)}
              \]

     \item If $\lambda,\mu\in M(\widehat{G})$, then $(\lambda\otimes \mu)*\mul\in \BM{X_{1},X_{2}}{X}{G}$ and
    \[
        \norm{(\lambda\otimes \mu)*\mul}_{\BM{X_{1},X_{2}}{X}{G}}\leq
        \mathfrak{d}\norm{\lambda}_{M(\widehat{G})}\norm{\mu}_{M(\widehat{G})}\norm{\mul}_{\BM{X_{1},X_{2}}{X}{G}},
    \]
with $\mathfrak{d}$ as above.
\end{enumerate}
\end{prop}
\begin{proof}
The fist assertion is almost direct, so we omit the proof. We shall prove the second one.  Observe first that for any $f,g\in SL^1(G)$,
\begin{equation}\label{eq:previo_minkowski_1}
    B_{(\lambda\otimes \mu)*\mul}(f,g)(x)=\iint_{\widehat{G}^2}  \esc{\zeta+\gamma, x} B_{\mul} (M_{-\zeta} f, M_{-\gamma} g)(x)\, {\dd \lambda(\zeta)}\, {\dd \mu(\gamma)},
\end{equation}
where $M_{\zeta} f(x)=\esc{\zeta, x} f(x)$.
If $X$ is Banach, the result follows by Minkowski's integral inequality. So, it remains to prove the case when $X$ is quasi-Banach and $p<1$. 

Assume first that there exists a compact set $\mathcal{K}$ such that $\lambda$ and $\mu$ are supported in $\mathcal{K}$. By the Lemma \ref{lem:technic_1}, there
exists a sequence of symmetric relatively compact open neighbourhood $\{U_n\}_n$ of $e_{\widehat{G}}$,
satisfying that $U_n+U_n\subset U_{n-1}$ and that for every $n\geq 1$ and $\zeta,\zeta^\prime\in U_n$,
\[
    \sup_{x\in G,\gamma\in \widehat{G}} \abs{F_{x}(\zeta,\gamma)-F_{x}(\zeta^\prime,\gamma)}<\frac 1 n.
\]
Since $\mathcal{K}$ is compact, there exists $N_n\in \N$, $\zeta_1,\ldots,
\zeta_{N_n}\in \mathcal{K}$ such that
\[
    \mathcal{K}\subset \bigcup_{j=1}^{N_n} U_{n}+\zeta_j.
\]
If we define, for $j=2,\ldots, N_n$,
\[
    \Omega_n^j=\brkt{U_n+\zeta_{j}}\setminus \Omega_n^{j-1},\quad \Omega_n^1=U_n+\zeta_1,
\]
we obtain a disjoint covering of $\mathcal{K}$, $\mathcal{K}\subset \biguplus _{j=1}^{N_n^\prime} \Omega_n^{j}$, with $N_n'\leq N_n$ such that
\[
     \sup_{\zeta\in \Omega_n^j} \sup_{x\in G, \gamma\in \widehat{G}} \abs{F_{x}(\zeta,\gamma)-F_{x}(\zeta_j,\gamma)}\leq \frac 1 n.
\]
By (\ref{eq:previo_minkowski_1}), it follows that
\begin{equation}\label{eq:desigualdad}
    \abs{B_{(\lambda\otimes \mu)*\mul}(f,g)(x)}\leq \iint_{\widehat{G}^2}\abs{F_x(\zeta,\gamma)}\; {\dd \abs{\lambda}(\zeta)}\, {\dd \abs{\mu}(\gamma)}.
\end{equation}
For any $n$, the inner integral in the right hand side can be bounded by
\[
    \begin{split}
    \int_{\widehat{G}} \abs{F_x(\zeta,\gamma)}\; {\dd \abs{\lambda}(\zeta)}&=\sum_{j=1}^{N_n'}  \int_{\Omega_n^j} \abs{F_x(\zeta,\gamma)}\; {\dd \abs{\lambda}(\zeta)}\\
    &\leq \frac{1}{n} \sum_{j=1}^{N_n'} \int_{\Omega_n^j}{\dd \abs{\lambda}(\zeta)}+\sum_{j=1}^{N_n'} \abs{F_{x}(\zeta_j,\gamma)}\int_{\Omega_n^j}{\dd \abs{\lambda}(\zeta)}\\
    &= \frac{\norm{\lambda}_{M(\widehat{G})}}{n}+\sum_{j=1}^{N_n'} \abs{F_{x}(\zeta_j,\gamma)} a_{n}^j,
    \end{split}
\]
where $a_{n}^j:=\int_{\Omega_n^j}{\dd \abs{\lambda}(\zeta)}\, {\dd \zeta}$. Then \eqref{eq:desigualdad} yields
\begin{equation}\label{eq:technic_12}
     \abs{B_{(\phi\otimes \psi)*\mul}(f,g)(x)}\leq 
     \frac{\norm{\lambda}_{M(\widehat{G})}\norm{\mu}_{M(\widehat{G})}}{n}+
     \sum_{j=1}^{N_n'} a_{n}^j \int_{\widehat{G}}\abs{F_{x}(\zeta_j,\gamma)}{\dd \abs{\mu}(\gamma)}.
\end{equation}
Repeating the argument with each integral appearing on the right hand side, we
can find a family of disjoint subsets $\{\Upsilon_n^k\}_{k=1}^{M_n}$, 
and a family $\{\gamma_k\}_{k=1}^{M_n}\subset \mathcal{K}$ satisfying
that $\mathcal{K}\subset \uplus \Upsilon_n^k$, and that
\[
    \sup_{\zeta \in \Upsilon_n^k}\sup_{\zeta\in \widehat{G}, x\in G} \abs{F_{x}(\zeta,\gamma)-F_{x}(\zeta,\gamma_k)}\leq \frac{1}{n}.
\]
In this way, if we define $b_{n}^k=\int_{\Upsilon_n^k}{\dd \abs{\mu}(\gamma)}$,
the sum in \eqref{eq:technic_12} is bounded by
\[
    \frac{\norm{\mu}_{M(\widehat{G})}}{n}\sum_{j=1}^{N_n} a_n^j +  \sum_{j=1}^{N_n'}\sum_{k=1}^{M_n} a_n^j b_{n}^k \abs{F_{x}(\zeta_j,\gamma_k)}.
\]
Thus, using that $\sum_{j=1}^{N_n} a_n^j=\norm{\lambda}_{M(\widehat{G})}$, we have
\[
   \iint_{\widehat{G}^2}\abs{F_x(\zeta,\gamma)}\; {\dd \abs{\lambda}(\zeta)}\, {\dd \abs{\mu}(\gamma)}\leq \frac{2 \norm{\lambda}_{M(\widehat{G})}\norm{\mu}_{M(\widehat{G})}}{n}+
    \sum_{j=1}^{N_n}\sum_{k=1}^{M_n} a_n^j b_{n}^k \abs{F_{x}(\zeta_j,\gamma_k)}.
\]
Cauchy-Schwarz inequality yields
\[
    \begin{split}
    \sum_{j=1}^{N_n} &\sum_{k=1}^{M_n} a_n^j b_{n}^k \abs{F_{x}(\zeta_j,\gamma_k)}\leq\\
    &\leq
    \sqrt{\norm{\lambda}_{M(\widehat{G})}\norm{\mu}_{M(\widehat{G})}} \brkt{\sum_{j=1}^{N_n}\sum_{k=1}^{M_n}
    \abs{B_{\mul}(\sqrt{a_n^j}M_{-\zeta_j} f, \sqrt{b_n^k }M_{-\gamma_k} g)(x)}^2}^{1/2}.
    \end{split}
\]
Since $B_\mul$ is a bounded bilinear operator, Lemma \ref{lem:MZ} implies
\[
    \begin{split}
    &\norm{\brkt{\sum_{j=1}^{N_n}\sum_{k=1}^{M_n} \abs{B_{\mul}(\sqrt{a_n^j}M_{-\zeta_j} f, \sqrt{b_n^k} M_{-\gamma_k} g)(x)}^2}^{1/2}}_{X}\leq\\
    &\leq \mathfrak{d}\norm{\mul} \norm{\brkt{\sum_j \abs{\sqrt{a_n^j}M_{-\zeta_j} f(x)}^2}^{1/2}}_{X_1} \norm{\brkt{\sum_k \abs{\sqrt{b_n^k}M_{-\gamma_k} g(x)}^2}^{1/2}}_{X_2}\\
    &= \mathfrak{d} \norm{\mul}_{\BM{X_{1},X_{2}}{X}{G}} \sqrt{\norm{\lambda}_{M(\widehat{G})}\norm{\mu}_{M(\widehat{G})}} \norm{f}_{X_1}\norm{g}_{X_2}.
    \end{split}
\]
So, for any compact set $\mathcal{R}\subset G$, and any $n\geq 1$, \eqref{eq:technic_12} yields
\begin{equation}
	\begin{split}
     \big\Vert B_{(\lambda\otimes \mu)*\mul}&(f,g) \chi_{\mathcal{R}}\big\Vert_X \leq 
     \norm{\chi_{\mathcal{R}}(x)\int_{\widehat{G}^2} \abs{F_x(u,\gamma)}\; {\dd \abs{\lambda}(u)}\, {\dd \abs{\mu}(\gamma)}}_X\\
     &\frac{2 \norm{\lambda}_{M(\widehat{G})}\norm{\mu}_{M(\widehat{G})} \norm{\chi_{\mathcal{R}}}_X}{n}+\\
     &+\mathfrak{d} \norm{\mul}_{\BM{X_{1},X_{2}}{X}{G}} {\norm{\lambda}_{M(\widehat{G})}\norm{\mu}_{M(\widehat{G})}} \norm{f}_{X_1}\norm{g}_{X_2}.
\end{split}
\end{equation}
Hence, taking first limit in $n\to \infty$ we have that for any compact set $\mathcal{R}\subset G$
\[
	 \big\Vert B_{(\lambda\otimes \mu)*\mul}(f,g) \chi_{\mathcal{R}}\big\Vert_X\leq \mathfrak{d} \norm{\mul}_{\BM{X_{1},X_{2}}{X}{G}} {\norm{\lambda}_{M(\widehat{G})}\norm{\mu}_{M(\widehat{G})}} \norm{f}_{X_1}\norm{g}_{X_2}.
\]
Taking a family of compact sets $\mathcal{R}_\uparrow G$ and using the monotonicity of the norm the result follows. 

For general case, consider an increasing sequence on compact sets  $\mathcal{K}_n\uparrow \widehat{G}$ . Monotone convergence implies
\[
    \abs{B_{(\lambda\otimes \mu)*\mul}(f,g)(x)}\leq \lim_{n} \iint_{\mathcal{K}_n\otimes \mathcal{K}_n}  \abs{F_x(\zeta,\gamma)}\; {\dd \abs{\lambda}(\zeta)}\, {\dd \abs{\mu}(\gamma)}.
\]
Using Fatou's property of $X$ and arguing as before, we obtain that for any compact set $\mathcal{R}\subset G$,
\[
	\begin{split}
    \norm{B_{(\phi\otimes \psi)*\mul}(f,g)\chi_\mathcal{R}}_{X}&\leq \mathfrak{d} \norm{\mul}_{\BM{X_{1},X_{2}}{X}{G}} \left(\liminf_n \int_{\mathcal{K}_n}\dd \abs{\lambda}\int_{\mathcal{K}_n}\dd \abs{\mu}\right) \norm{f}_{X_1}\norm{g}_{X_2},\\
    &\leq \mathfrak{d} \norm{\mul}_{\BM{X_{1},X_{2}}{X}{G}}{\norm{\lambda}_{M(\widehat{G})}\norm{\mu}_{M(\widehat{G})}} \norm{f}_{X_1}\norm{g}_{X_2}.
    \end{split}
\]
Arguing as before, the monotone convergence theorem yields the result. 
\end{proof}

Having proved the previous result we are now in a position to continue the proof of Theorem \ref{thm:Main}. So we need to give the countable family of multipliers $\{\mul_j\}_j$ satisfying$ (P_1)$-$(P_4)$. To this end, let consider $\varphi_j\in \mathcal{C}_c(G)$ such that
    \begin{enumerate}
        \item[($I_1$)]  For every $j\geq 0$, $\widehat{\varphi_j}\geq 0$;
        \item[($I_2$)] For every $j\geq 0$, $\int_{\widehat{G}} \widehat{\varphi_j}=1$;
        \item[($I_3$)] For every relatively compact open set $\mathcal{K}\subset \widehat{G}$ such that $e_{\widehat{G}}\in \mathcal{K}$,
        \[
            \lim_{j} \int_{\xi\not\in \mathcal{K}} \widehat{\varphi_j}=0.
        \]
    \end{enumerate}
    In other words, $\{\widehat{\varphi}_j\}_j$ is an approximate identity for $L^1(\widehat{G})$, which existence is ensured by \cite{MR0336233}*{Lemma 3.4}.
    Consider $\Phi_j=\varphi_j\otimes \varphi_j\in\mathcal{C}_c(G^2)$. It is easy to see that $\{\widehat{\Phi}_j\}_j$ is an approximate identity for $L^1(\widehat{G}^2)$.
    
        Consider $h_j\in \mathcal{C}_c(G)$ such that $0\leq h_j\leq 1$, $\int h_j=1$ and such that, for any $\xi\in \widehat{G}$ $\lim_j \widehat{h_j}(\xi)=1$.
    Define
    \begin{equation}\label{eq:mjs}
        \mul_j=\widehat{\left(h_j\otimes h_j\right)} \left(\widehat{(\varphi_j\otimes \varphi_j)}*\mul\right)
    \end{equation}
A similar argument to \cite{MR0336233}*{Lemma 3.5} for the group $G\otimes G$, implies properties ($P_1$) and ($P_3$) for $\{\mul_j\}_j$. On the other hand, since $X_1$ and $X_2$ are RI BFS, Minkowski integral inequality yields that, for any $j$, $\widehat{h_j}\in \mathcal{M}_{X_k}(G)$ and $\norm{h_j}_{\mathcal{M}_{X_k}(G)}\leq 1$ for $k=1,2$. Thus, Proposition \ref{prop:clave} yields that the sequence $\{\mul_j\}_j$ satisfies ($P_4$). 

In order to finish the proof of Theorem \ref{thm:Main} we need to prove ($P_2$) and ($P_2^\prime$). We are going first to recall the concept of normalized function \cite{MR0481928}*{Chapter 3}.

\begin{defn}\label{def:normalized}We say that $\mul\in L^\infty(\widehat{G}^2)$ is a normalized function (with respect to $\Phi_j$) if,
for any $\xi,\eta\in \widehat{G}$,
\[
    \lim_j \mul*\Phi_j(\xi,\eta)=\mul(\xi,\eta).
\]
\end{defn}

It follows from properties $(I_1)$, $(I_3)$, $(I_3)$ above, that if $(\xi,\eta)\in \widehat{G}^2$ is a continuity point of $\mul$, $ \lim_j \mul*\Phi_j(\xi,\eta)=\mul(\xi,\eta)$. That is, if $\mul\in \mathcal{C}_b(\widehat{G}^2)$, then it is a normalized function (with respect to $\{\widehat{\varphi_j}\otimes\widehat{\varphi_j}\}$). Above all, if $\mul\in \mathcal{C}_b(\widehat{G})$, then the sequence $\set{\mul_j}_j$ given in \eqref{eq:mjs} satisfies ($P_2^\prime$). 

\begin{rem}\label{rem:normalizedHT} If $\mul(\xi,\eta)=M(\eta-\xi)$ where $M\in L^\infty(\widehat{G})$, then it is easy to see that
\[
    \mul*(\phi\otimes\psi)(\xi,\eta)=M*_1 (\psi*_1 \widetilde{\phi})(\eta-\xi).
\]
where $*_1$ indicates the convolution for functions in $\widehat{G}$
and $\widetilde{\phi}(z)=\phi(-z)$. Therefore, if $M$ is a normalized function on $\widehat{G}$ with respect to $\{\varphi_j\}$, so it is $\mul$  on $\widehat{G}^2$ with respect to $\{\varphi_j\otimes \varphi_j\}$.  That is the case, for instance, of the function $\mul(\xi,\eta)=-i\sign(\eta-\xi)$, which is the multiplier associated to the Bilinear Hilbert Transform.
\end{rem}

We have proved that $\set{\mul_j}_j$ defined in \eqref{eq:mjs} satisfies ($P_1$),($P_3$),($P_4$) and observe that, any partial sequence also does.  Then, for the general case, it suffices to ensure the a.e. convergence property for a partial sequence of   $\set{\mul_j}_j$. To this end, we need the following technical lemma.
\begin{lem} Let $\Gamma$ be a LCA group and let $\{\widehat{\Phi_j}\}_j$ be an approximate identity for $L^1\brkt{\Gamma}$ and let $\mathbf{b}\in L^\infty\brkt{\Gamma}$. Define $\mathbf{b}_j=\widehat{\Phi_j}*\mathbf{b}$. Then, there exist a partial sequence $\{\mathbf{b}_{j_k}\}_k$ such that 
\[
	\lim_k \mathbf{b}_{j_k}(\xi)=\mathbf{b}(\xi)\quad {\rm a.e.}\, \xi\in \Gamma.
\]
\end{lem}
\begin{proof}
Suppose first that $\Gamma$ is a compact group. Then $L^\infty(\Gamma)\subset L^1(\Gamma)$ and, since $\widehat{\Phi_j}$ is an approximate identity, $\lim_j \mathbf{b}_j=\mathbf{b}$ in the $L^1(\Gamma)$ norm.  In particular, there exits a partial sequence of $\{\mathbf{b}_j\}_j$, such that we have the desired a.e. convergence.

Suppose now that $\Gamma$ is a non-compact group. Let $H_n$ be a sequence of relatively compact, symmetric  open neighbourhoods of the identity element in $\Gamma$, such that $H_n\subset H_n+H_n\subset H_{n+1}$ and $\Gamma=\cup_n H_n$. Observe that this family $\{H_n\}_n$ satisfies that for any $n\geq 1$, there exists $m(n)>n$ such that $H_n+(\Gamma\setminus H_{m(n)})\subset (\Gamma\setminus H_n)$. Note that any partial sequence of $\{\Phi_j\}$ is also an approximate identity for $L^1(\Gamma)$. 

Then, fixed $n$, since $\mathbf{b}\chi_{H_{m(n)}}\in L^1(\Gamma)$, $\chi_{H_n}(\widehat{\Phi_j}*\mathbf{b}\chi_{H_{m(n)}})$ converges to $\mathbf{b}\chi_{H_n}$ in the $L^1$ norm, when $j$ tends to infinity. On the other hand, for any $\xi\in H_n$,
\[
	\abs{\widehat{\Phi_j}*\mathbf{b}\chi_{\Gamma\setminus H_{m(n)}}(\xi)}\leq  \norm{\mathbf{b}}_{\infty}\int_{H_n+\Gamma\setminus H_{m(n)}} \widehat{\Phi_j}(\eta)\, \dd \eta\leq 
	\norm{\mathbf{b}}_{\infty}\int_{\Gamma\setminus H_{n}} \widehat{\Phi_j}(\eta)\, \dd \eta,
\]
which converges to zero when $j$ tends to infinity. 

Then, by an induction argument we can construct a partial sequence $\{\mathbf{b}_{j_k}\}_k$ satisfying that, for any $n\geq 1$, there exists a set of measure zero $N_n$, such that, for any $\xi \in H_n\setminus N_n$, $\lim_n \mathbf{b}_{j_k}(\xi)=\mathbf{b}(\xi)$. A standard measure argument yields the desired result. 
\end{proof}

For a general $\mul\in L^\infty(\widehat{G}^2)$, the previous  lemma with $\Gamma=\widehat{G}^2$ provide us with a partial sequence $\{\widehat{\Phi_{j_k}}*\mul\}_k$ which satisfies ($P_2$).  In particular, $\{\mul_{j_k}\}_k$, which is a partial sequence of that given in \eqref{eq:mjs}, that we rename as the new $\{\mul_j\}_j$,  satisfies ($P_1$)-($P_4$). 

\section{Proof of Theorem \ref{thm:main theorem}}\label{sect:approx}
Lets consider the case $\mul\in \BM{p_1,p_2}{p}{\Gamma}$. The weak case is proved analogously, so we omit the details. We want to prove that 
 $\mul\circ (\pi\otimes \pi)\in
\BM{p_1,p_2}{p}{G}$ and
\[
    \norm{\mul\circ (\pi\otimes \pi)}_{\BM{p_1,p_2}{p}{G}}\leq \mathfrak{c} \norm{\mul}_{\BM{p_1,p_2}{p}{\Gamma}}.
\]

Assume first that there exists $K\in L^1_c(\Gamma^2)$ such that $\widehat{K}=\mul$. In this case, it is easy to see that the multiplier operator coincides with the operator given by
\[
    B_K(F,G)(x)=\iint_{\Gamma^2} F(x-y_1) G(x-y_2) K(y_1,y_2)\, {\dd y_1}\, {\dd y_2},
\]
which by assumption on $\mul$, is a bounded operator from $L^{p_1}(\Gamma)\times L^{p_2}(\Gamma)$ to $L^{p}(\Gamma)$ with bound $\norm{\mul}_{\wBM{p_1,p_2}{p}{\Gamma}}$.

Let $\widetilde{\pi}:\Gamma\to G$ be the dual homomorphism of
$\pi$ defined by
\[
    \esc{x,\widetilde{\pi}(z)}=\esc{\pi(x),z},\quad \forall x\in G,\, \forall z\in \Gamma,
\]
that, by \cite{MR551496}*{(24.38)} it is a continuous homomorphism, which induces a strongly continuous, measure preserving representation of $\Gamma$ in $L^{q}(G)$ for any $0<q<\infty$ given by
\[
	R_z f_1(x)=f_1(\widetilde{\pi}(z)+x).
\]   
This representation satisfies, for any $z_0,z_1,z_2\in \Gamma$,
\[
    R_{z_0} \left(R_{z_1} f_1 R_{z_2} f_2\right)=R_{z_0+z_1} f_1 \, R_{z_0+z_2} f_2,
\]
and that for any $z\in \Gamma$, $\norm{R_{z} f}_{L^{q}}=\norm{f}_{L^{q}}$ for $q=p,p_1,p_2$.

Consider the {\it Transferred operator} as in  \cite{MR1398100},  given by
\[
	 T_K(f_1,f_2)(x)=\iint_{\Gamma^2} K(z_1,z_2) R_{z_1}f_1(x)  R_{z_2} f_2(x) \; \dd z,
\]
 for $f_1,f_2\in SL^1(G)$. Then \cite{MR1398100}*{Theorem 1} ( \cite{MR1398100}*{Theorem 2} for the weak case) yields that $T_K$ can be extended to a bounded operator $L^{p_1}(G)\times L^{p_2}(G)$ to $L^p(G)$ with a bound no larger than $\norm{\mul}_{\BM{p_1,p_2}{p}{\Gamma}}$. But observe that it holds that 
\begin{eqnarray*}
    \lefteqn{{T}_K(f_1,f_2)(x)=\iint_{\Gamma^2} K(z_1,z_2) R_{z_1}^1f_1(x) R_{z_2}^2f_2(x)\, {\dd z_1}\, {\dd z_2}}\\
    &=&\iint_{\widehat{G}^2} \widehat{f_1}(\xi)\widehat{f_2}(\eta) \iint_{\Gamma^2} K(z_1,z_2)\esc{\xi,\widetilde{\pi}(z_1) +x}\esc{\eta,\widetilde{\pi}(z_2)+x}\, {\dd z_1}\, {\dd \gamma_2}\, {\dd \xi} {\dd \eta}\\
    &=&\iint_{\widehat{G}^2} \widehat{f_1}(\xi)\widehat{f_2}(\eta)\widehat{K}(\pi(\xi),\pi(\eta))\esc{\eta+\xi, x} {\dd \xi}\, {\dd \eta}\\
    &=&B_{\mul\circ (\pi\otimes \pi)}(f_1,f_2)(x),
\end{eqnarray*}
which yields  
\begin{equation}
	\label{eq:cpt_support} \norm{\mul\circ (\pi\otimes \pi)}_{\BM{p_1,p_2}{p}{G}}\leq \norm{\mul}_{\wBM{p_1,p_2}{p}{\Gamma}},
\end{equation}
for $\mul\in \widehat{L^{1}_c(\Gamma^2)}$.

Lets assume now that $\mul\in \mathcal{C}_b(\widehat{\Gamma}) \cap\BM{p_1,p_2}{p}{\Gamma}$.  Let $\{\mul_j\}_j$ be the sequence given by Theorem \ref{thm:Main}.  The Dominated convergence theorem, ($P_2^\prime$) and  ($P_3$) imply that,  for any $f_1,f_2\in SL^1(G)$,
\begin{equation}\label{eq:improvement}
    \begin{split}
       {B}_{\mul\circ (\pi\otimes \pi)}&(f_1,f_2)(x) =\iint_{\widehat{G}^2} \widehat{f_1}(\xi)\widehat{f_2}(\eta) \mul(\pi(\xi),\pi(\eta)) \esc{\xi+\eta,x}\, {\dd \xi} {\dd \eta}\\
        &=\lim_{j}\iint_{\widehat{G}^2} \widehat{f_1}(\xi)\widehat{f_2}(\eta) \mul_j(\pi(\xi),\pi(\eta)) \esc{\xi+\eta,x}\, {\dd \xi} {\dd \eta}.
    \end{split}
\end{equation}
Then, Fatou's lemma, ($P_1$) and \eqref{eq:cpt_support} yield
\[
    \begin{split}
    \norm{{B}_{\mul\circ (\pi\otimes \pi)}(f_1,f_2)}_{L^{p}(G)}&\leq \liminf_j \norm{{B}_{\mul_j\circ (\pi\otimes \pi)}(f_1,f_2)}_{L^{p}(G)}\\
    &\leq \liminf_j \norm{\mul_j}_{\BM{p_1,p_2}{p}{\Gamma}} \norm{f_1}_{L^{p_1}(G)}\norm{f_2}_{L^{p_2}(G)}.
    \end{split}
\]
Therefore, the results follows as by ($P_4$), 
\[
    \liminf_j \norm{\mul_j}_{\BM{p_1,p_2}{p}{\Gamma}}\leq \mathfrak{c} \norm{\mul}_{\BM{p_1,p_2}{p}{\Gamma}}.
\]
\begin{flushright}
\qedsymbol
\end{flushright}

\begin{rem}\label{eq:final_remark} Observe that in the proof of the previous theorem, \eqref{eq:improvement} holds, and also the statement of Theorem \ref{thm:main theorem} does, if we can ensure that a.e. $(\xi,\eta)\in G^2$, the limit  $\lim_j \mul_j(\pi(\xi),\pi(\eta))=\mul(\pi(\xi),\pi(\eta))$ holds.  Hence, by Remark \ref{rem: normalized}, the result holds if $\mul$ is normalized or it is continuous on the image of $G^2$ by $\pi\otimes \pi$. 
\end{rem}

\section{Application and consequences}\label{sect:Application and consequences}

\subsection{Applications of Theorem \ref{thm:main theorem}}
In this section we will restrict our attention to indices $p_1,p_2,p$ satisfying $1\leq p_1,p_2<\infty$, $0<p<\infty$ such that 
\begin{equation}\label{eq:relacion}
    \frac{1}{p_1}+\frac{1}{p_2}=\frac{1}{p}.
\end{equation}

\subsubsection{K. de Leeuw's restriction type results}
Here we show how our Theorem \ref{thm:main theorem} allows us to produce de Leeuw's type bilinear results.  For the sake of brevity, we restrict our results only the strong case, but it has to be kept in mind that the corresponding weak results also hold.  

We shall start with an abstract version of of D. Fan and S. Sato's \cite{MR1808390}*{Theorem 3} to LCA groups. 
Let $G$ be a LCA group and let $H$ be a closed subgroup of $G$. Consider  $\Gamma = G/H$. Recall that the dual group of $\Gamma$ can be identified as $$\widehat{\Gamma}= H^{\perp}:=\set{\xi\in \widehat{G}:\, \forall g\in H\, \esc{\xi,g}=1}.$$ Letting be $\pi$ is the canonical inclusion of $H^{\perp}\hookrightarrow \widehat{G}$, and $\Pi$ the canonical projection from $\widehat{G}\to \widehat{H}$, Theorem \ref{thm:main theorem} yields the following abstract result.
\begin{thm}\label{thm:DeLeeuw} If $G$ is a LCA groups and $H$ is a closed subgroup. Then, 
\begin{enumerate}
	\item  If $\mul\in \BM{p_1,p_2}{p}{G}\cap \mathcal{C}_b(\widehat{G}^2)$ then $\mul(\pi\otimes\pi) \in \BM{p_1,p_2}{p}{G/H}$ and
	\[
		\norm{\mul(\pi\otimes \pi)}_{\BM{p_1,p_2}{p}{G/H}}\leq \mathfrak{c}\norm{\mul}_{\BM{p_1,p_2}{p}{G}}.
	\]
		\item\label{thm:deLeeuw_2} If $\mul\in \BM{p_1,p_2}{p}{H}\cap \mathcal{C}_b(\widehat{H}^2)$ then $\mul(\Pi\otimes\Pi) \in \BM{p_1,p_2}{p}{G}$ and
	\[
		\norm{\mul(\Pi\otimes \Pi)}_{\BM{p_1,p_2}{p}{G}}\leq \mathfrak{c}\norm{\mul}_{\BM{p_1,p_2}{p}{H}}.
	\]
\end{enumerate}
\end{thm}

In the particular case $G=\R^d$, $H=\Z^d$, identifying $\T^d$ with $[0,1)^d$, if we consider $\pi$ to be the canonical projection $\pi(\xi)=(\xi_1-[\xi_1],\ldots,\xi_d-[\xi_d])$, where $[t]$ denotes the integer part of $t$, the previous result implies the following.

\begin{cor}  Let $\mul\in \mathcal{C}_b(\T^{2d})\cap \BM{p_1,p_2}{p}{\Z^d}$. If we define
$\widetilde{\mul}(\xi,\eta)=\mul(\xi_1-[\xi_1],\eta_1-[\eta_1],\ldots,\xi_d-[\xi_d],\eta_d-[\eta_d])$,  then
$
    \widetilde{\mul}\in \BM{p_1,p_2}{p}{\R^d}
$
and
\[
    \norm{\widetilde{\mul}}_{\BM{p_1,p_2}{p}{\R^d}}\leq \mathfrak{c} \norm{\mul}_{\BM{p_1,p_2}{p}{\Z^d}}.
\]
\end{cor}

Observe that in Theorem \ref{thm:main theorem}, the obtained bound does not depend on the homomorphism considered. This allows us to obtain a extension of D. Fan and S. Sato's \cite{MR1808390}*{Theorem 3}. 

\begin{cor}\label{cor:Fan_Sato} Let $\mul\in \mathcal{C}_b(\R^{2d})\cap \BM{p_1,p_2}{p}{\R^{d}}$.
For any ${\vec\epsilon}=(\epsilon_1,\ldots, \epsilon_d)\in \brkt{\R_+}^d$ let $\pi_{\vec\epsilon}:\Z^d\to \R^d$ be the anisotropic dilations $\pi_{\vec\epsilon}(n)=(\epsilon_1 n_1,\ldots, \epsilon_d n_d)$.  Then 
\[
    \sup_{{\vec\epsilon}\in \brkt{\R_+}^d}\norm{\mul_{\vec\epsilon}}_{\BM{p_1,p_2}{p}{\T^d}}\leq \mathfrak{c} \norm{\mul}_{\BM{p_1,p_2}{p}{\R^d}},
\]
where $\mul_{\vec\epsilon}(n,m)=\mul\brkt{\pi_{\vec\epsilon}(n),\pi_{\vec\epsilon}(m)}$, for $n,m\in \Z^d$.
\end{cor}

\begin{cor}\label{cor:several}  Let $\mul\in \mathcal{C}_b(\R^{2n})\cap\BM{p_1,p_2}{p}{\R^n}$.
Let $A: \R^d\to \R^{n}$ be a linear map given by a matrix $A$.
Define for $\xi^\prime,\eta^{\prime}\in \R^{d}$
\[
    \widetilde{\mul}(\xi^\prime,\eta^\prime )=\mul(A\xi^\prime,A\eta^\prime ).
\]
Then
$
    \widetilde{\mul}\in \BM{p_1,p_2}{p}{\R^d}
$
and
\[
    \norm{\widetilde{\mul}}_{\BM{p_1,p_2}{p}{\R^d}}\leq \mathfrak{c} \norm{\mul}_{\BM{p_1,p_2}{p}{\R^n}}.
\]
\end{cor}
\begin{proof}  Let $G=\R^n$, let $H=\set{\xi\in\R^n: \xi=A x,\, x\in \R^d}$ be the image of $A$ and apply \eqref{thm:deLeeuw_2} in Theorem \ref{thm:DeLeeuw}.
\end{proof}
The previous result allows us to obtain a generalization of G. Diestel and L. Grafakos's  \cite{MR2301463}*{Proposition 2} for $p<1$ (and also its weak type counterpart), on the restriction to a lower dimension of a bilinear multiplier. 
\begin{cor}  Let, for $n\geq 2$, $\mul\in \mathcal{C}_b(\R^{2n})\cap\BM{p_1,p_2}{p}{\R^n}$. Let $d<n$. Consider $d_1=n-d$.
For any $\eta_1,\eta_2\in \R^{d_1}$, the function defined by
\[
    \widetilde{\mul}(\xi_1,\xi_2)=\mul(\xi_1,\eta_1,\xi_2,\eta_2), \quad\forall \xi_1,\xi_2\in \R^{d},
\]
satisfies that $\widetilde{\mul}\in \BM{p_1,p_2}{p}{\R^d}$ and
\[
    \norm{\widetilde{\mul}}_{\BM{p_1,p_2}{p}{\R^d}}\leq \mathfrak{c} \norm{\mul}_{\BM{p_1,p_2}{p}{\R^n}}.
\]
\end{cor}
\begin{proof}
It is easy to see that if $\mul\in\BM{p_1,p_2}{p}{\R^n}$, for
any $\gamma_1,\gamma_2\in \R^{n}$, then the function $\mul_{\gamma_1,\gamma_2}(\gamma,\nu)=\mul(\gamma+\gamma_1,\nu+\gamma_2)\in\BM{p_1,p_2}{p}{\R^n}$ and
\[
    \norm{\mul_{\gamma_1,\gamma_2}}_{\BM{p_1,p_2}{p}{\R^n}}=\norm{\mul}_{\BM{p_1,p_2}{p}{\R^n}}.
\]
In particular, if we consider $\gamma_j=(0,\eta_j)$, for $j=1,2$ and we take the linear map $A:\R^d\to \R^n$ given by $A\xi=(\xi,0)$,
the result follows by the previous one applied to $\mul_{\gamma_1,\gamma_2}$ as
\[
    \widetilde{\mul}(\xi_1,\xi_2)=\mul_{\gamma_1,\gamma_2}(A\xi_1,A\xi_2).
\]
\end{proof}

We can also obtain the following two lifting results on multipliers. 
\begin{cor}  Let, $d>n\geq 1$ and $\mul\in \mathcal{C}_b(\R^{2n})\cap\BM{p_1,p_2}{p}{\R^n}$. Define, for any $(\xi_j,\eta_j) \in \R^{n}\times\R^{d-n}$ for $j=1,2$,
\[
    \widetilde{\mul}(\xi_1,\eta_1, \xi_2,\eta_2)=\mul(\xi_1,\xi_2).
\]
Then $\widetilde{\mul}\in \BM{p_1,p_2}{p}{\R^d}$ and
\[
    \norm{\widetilde{\mul}}_{\BM{p_1,p_2}{p}{\R^d}}\leq \mathfrak{c} \norm{\mul}_{\BM{p_1,p_2}{p}{\R^n}}.
\]
\begin{proof}
It suffices to consider the natural projection $A:\R^d\to \R^n$ such that maps any $(\xi,\eta)\in \R^{n}\times \R^{d-n}$ to $\xi$, and
apply Corollary \ref{cor:several} as
\[
     \widetilde{\mul}(\xi_1,\eta_1, \xi_2,\eta_2)=\mul (A(\xi_1,\eta_1), A(\xi_2,\eta_2))=\mul(\xi_1,\xi_2).
\]
\end{proof}
\end{cor}

\begin{cor} Let $\mul\in \mathcal{C}_b(\R^{2})\cap\BM{p_1,p_2}{p}{\R}$ and $d\geq 1$. Fixed $y\in \R^d$, define, for any $(\xi,\eta) \in \R^{2d}$ 
\[
    \widetilde{\mul}_y(\xi,\eta)=\mul(\xi\cdot y,\eta\cdot y).
\]
Then $\widetilde{\mul}_y\in \BM{p_1,p_2}{p}{\R^d}$ and
\[
    \sup_{y\in \R^n}\norm{\widetilde{\mul}_y}_{\BM{p_1,p_2}{p}{\R^d}}\leq \mathfrak{c} \norm{\mul}_{\BM{p_1,p_2}{p}{\R}}.
\]
\begin{proof}
It suffices to consider the linear form given by the scalar product $A_y x=x\cdot y$ and 
apply Corollary \ref{cor:several}.
\end{proof}
\end{cor}
%\begin{rem} In the particular case that
%$\mul(\xi,\eta)=M(\eta-\xi)$,
%\[
%    B_{\mul*(\delta_0\otimes\psi)}(f,g)=\int_{\widehat{G}^2} \psi*_1M(\eta-\xi) \widehat{f}(\xi) 
%    \widehat{g}(\eta) \dd \xi\dd\eta
%\]
%where $*_1$ indicates the usual convolution for functions in one
%variable.
%
%In the particular case that $M(t)=-i\sign t$, if we consider for
%$\epsilon<1<\delta$, $\psi_{\epsilon,\delta}(x)=\phi(\epsilon
%x)-\phi(x \delta)$, where $\phi\in \mathcal{C}^\infty(\R)$, $0\leq
%\phi\leq 1$, it is radial, $\phi(x)=1$ for $\abs{x}\leq 1$ and
%$\phi(x)=0$ if $\abs{x}\geq 1+\eta$. It follows that, for $p>2/3$,
%$\widehat{\psi_{\epsilon,\delta}}*M\in \BM{p_1,p_2}{p}{\R}$ and
%$\norm{\widehat{\psi_{\epsilon,\delta}}*M}\leq
%2\norm{\widehat{\phi}}_1 \norm{M}_{\BM{p_1,p_2}{p}{\R}}$. In other
%words, we can see that if
%\[
%    K_{\epsilon,\delta}(x)=\frac{\psi_{\epsilon,\delta}(x)}{x},
%\]
%$\{\widehat{K_{\epsilon,\delta}}\}_{\epsilon<1<\delta}$ define an
%uniformly bounded family of bilinear multipliers.
%\end{rem}
%
\subsubsection{Bilinear Hilbert transform in groups with ordered dual}
In this section, following the spirit of \cite{MR930884}, we define a Generalised Hilbert transform on groups with ordered dual by using the original version in $\R$, and obtain its boundedness. 

To this end we shall assume that, $G$ is a LCA group such that $\widehat{G}$ has a measurable order $P$.  That is, there exists $P\subset \widehat{G}$ measurable satisfying $P+P=P$, $P\cap (-P)=\{0\}$; $P\cup (-P)=\widehat{G}$.  The group $G=\T$ is an example of such class of groups (see \cite{MR930884} and the references therein for more information on ordered groups). With $P$ we associate the function ${\rm \sign}_P$ given by 
\[
	{\rm \sign}_P(\xi)=\begin{cases}
	1 & {\rm if}\quad  \xi\in P\setminus\{0\};\\
	0 & {\rm if}\quad  \xi =0;\\
	-1 & {\rm if}\quad  \xi\in (-P)\setminus\{0\}.
	\end{cases}
\]
\begin{defn} We define the Generalised Bilinear Hilbert Transform in $G$ by the operator given by the multiplier $\mul(\xi,\eta)=-i{\rm \sign}_P(\eta- \xi)$. That is, it is given by the expression
\[
	\mathcal{H}_G(f,g)(x)=\iint_{\widehat{G}^2} -i{\rm sign}_P(\eta-\xi)\widehat{f}(\xi)\widehat{g}(\eta)\esc{\xi+\eta,x}\, \dd \xi \dd \eta.
\]
\end{defn}
\begin{thm}\label{thm:generalisedBHT} With the notations as above, there exists a constant $C$ such that for any $f\in L^{p_1}(G)$, $g\in L^{p_2}(G)$,
\[
	\norm{\mathcal{H}_G(f,g)}_{p}\leq C \norm{f}_{p_1}\norm{g}_{p_2},
\]
provided $\frac{2}{3}<p<\infty$, $1\leq p_1,p_2<\infty$. 
\end{thm}
\begin{proof} By density, it is enough to prove the result for $f,g$ such that the support of $\widehat{f}$, $\widehat{g}$ is compact, with constants independently on these supports. Let $\mathcal{K}_f, \mathcal{K}_g$ be the support of $\widehat{f}$ and $\widehat{g}$ respectively. By \cite{MR930884}*{Theorem (5.14)}, there exists a homomorphism $\pi$ from $\widehat{G}$ to $\R$ such that the equality 
\[
	\sign_P(\xi)=\sign(\pi(\xi))
\]
holds for a.e. $\xi\in \mathcal{K}_g-\mathcal{K}_f$. Thus, since $\pi$ is an homomorphism
\[
	\sign_P(\eta-\xi)=\sign(\pi(\eta-\xi))=\sign(\pi(\xi)-\pi(\eta)).
\]
Hence, since by Remark \ref{rem:normalizedHT} $\sign$ is a normalized multiplier, we can apply Theorem \ref{thm:main theorem} (see Remark \ref{rem: normalized}), jointly with  Lacey and Thiele's results in \cite{MR1689336} to conclude the proof. 
\end{proof}
\subsubsection{Isomorphic groups} 
The following result is an immediate consequence of applying Theorem \ref{thm:main theorem} twice. 
\begin{cor}\label{cor:Isomorphism} If $G,\Gamma$ are LCA groups that are topologically isomorphic, then the spaces  $\BM{p_1,p_2}{p}{\Gamma}\cap \mathcal{C}_b(\Gamma)$ and $\BM{p_1,p_2}{p}{G}\cap  \mathcal{C}_b(G)$ ($\wBM{p_1,p_2}{p}{\Gamma}\cap \mathcal{C}_b(\Gamma)$ and $\wBM{p_1,p_2}{p}{G}\cap  \mathcal{C}_b(G)$ respectively), are isomorphic.
\end{cor}

In particular the previous corollary and  \cite{MR551496}*{Theorem (9.8)} imply that, if $G$ is a LCA group such that $\widehat{G}$ is a compactly generated LCA (see \cite{MR551496}*{Definition (5.12)}), then $\BM{p_1,p_2}{p}{G}\cap  \mathcal{C}_b(\widehat{G})$ is fully characterised by the space $\BM{p_1,p_2}{p}{\Gamma}\cap  \mathcal{C}_b(\widehat{\Gamma})$ where 
$\Gamma$ is a LCA group of the type $\R^a\times \T^b\times K$,  where $a,b$ are non-negative integers and $K$ is a discrete Abelian group.  

\subsection{Consequences of Theorem \ref{thm:Main} }

As a result of the approximation theorem, we will obtain also, a  necessary condition (analogous to Hormander's \cite[Theorem 1.1]{MR0121655}), which generalises Grafakos and Torres's \cite[Proposition 5]{MR1880324}.  We shall first recall the concept of Boyd indices of a RI QBFS. 

\begin{defn}(see \cite{MR928802}) For any RI QBFS $X$, the
upper Boyd index is defined by
\begin{equation}\label{eq:boydsup}
    \overline{\alpha}_X=\inf\set{p:\ \exists c\, \forall a>1\, h_X(a)\leq c
    a^p},
\end{equation}
and the lower Boyd index by
\begin{equation}\label{eq:boydinf}
    \underline{\alpha}_X=\sup\set{p:\ \exists c\, \forall a<1\, h_X(a)\leq c
    a^p}.
\end{equation}
\end{defn}

We shall mention that, for Lorentz spaces $X=L^{p,q}$, and in particular for $L^{p}$ spaces, $\underline{\alpha}_X= \overline{\alpha}_X=\frac{1}{p}$. 
\begin{thm}\label{thm:necessity} Let $G$ be a non-compact LCA group and let $X_1,X_2, X$ be RI QBFSs on $G$ such that that $X$ is the $p$-convexification of a RI BFS and $X_j$ is a $p_j$-concave RI BFS, for $j=1,2$.  If there exists $\mul\in \BM{X_{1},X_{2}}{X}{G}$, $\mul\neq 0$, then
\[
    \underline{\alpha}_{X}\leq  \overline{\alpha}_{X_1}+\overline{\alpha}_{X_2}.
\]
\end{thm}
\begin{proof}
Observe that by Theorem \ref{thm:Main}, we can reduce ourselves to the case that  there exists $K\in L^1_c(G^2)\setminus \{0\}$ such that $\widehat{K}=\mul$.

        Let $\mathcal{K}_0$ be a symmetric compact neighbourhood of $e_g$ in $G$ such that
    $\supp K\in \mathcal{K}_0\times \mathcal{K}_0$. Observe that if $f_1,f_2$ are functions in $SL^1(G)$, supported in compact sets $L_1$ and $L_2$ respectively, then the operator 
    \[
    	B_\mul(f_1,f_2)(x)=\iint_{G^2} K(u_1,u_2) f_1(x-u_1)f_2(x-u_2)\, \dd u,
    \]
    is supported in the compact set $\brkt{\mathcal{K}_0+L_1}\cap \brkt{\mathcal{K}_0+L_2}$. 
    
    Let $\mathcal{K}$ be a compact neighbourhood of $e_g$ and let $f,g\in SL^1(G)$ with
    support in $\mathcal{K}$ such that $\norm{B_K(f,g)}_{X}> 0$. Observe that 
    \[
        \supp B_{\mul}(f,g)\subset\mathcal{K}_0+\mathcal{K}.
    \]
   Consider the translation operator $\tau$ given by $\tau_y g(x)=g(x-y)$. Since $G$ is  not compact, there exists a sequence $\{y_j\}_{j\geq 0}$ of elements of $G$,  with $y_0=e_g$, such that the compact sets $\{\mathcal{K}_0+\mathcal{K}+y_j\}_{j\geq 0}$ are pairwise disjoint.  It follows that  for any pair of indices $j\neq k$
    \begin{eqnarray}
            \label{eq:tech_1}B_{\widehat{K}}(\tau_{y_j}f,\tau_{y_k}g)&=& 0;\\
            \label{eq:tech_2}\tau_{y_j}f \tau_{y_k}f &=&0;\\
             \label{eq:tech_4}\tau_{y_j}g \tau_{y_k}g &=&0;\\
            \label{eq:tech_3}\tau_{y_j}B_{\widehat{K}}(f,g) \tau_{y_k}B_{\widehat{K}}(f,g)
            &=&  0.
    \end{eqnarray}
    Thus, for any $N\geq 1$, bilinearity and \eqref{eq:tech_1} yield
    \[
        \sum_{k=0}^N
        \tau_{y_k}B_{\widehat{K}}(f,g)(x)=\sum_{k=0}^N\sum_{j=0}^N
        B_{\widehat{K}}(\tau_{y_j}f,\tau_{y_k}g)(x)=B_{\widehat{K}}(\sum_{j=0}^N\tau_{y_j}f,\sum_{k=0}^N\tau_{y_k}g)(x).
    \]
    Then \eqref{eq:tech_3} yields
    \[
        (N+1)\mu\set{x\in G:\, \abs{B_{\widehat{K}}(f,g)(x)}>s}= \mu\set{x\in G:\, \abs{\sum_{k=0}^N
        \tau_{y_k}B_{\widehat{K}}(f,g)(x)}>s},
    \]
    which implies,  
    \[
        \norm{E_{{N+1}}\brkt{B_{\widehat{K}}(f,g)}^*}_{X^*}\leq
        \norm{\widehat{K}} \norm{\sum_{j=0}^N\tau_{y_j}f}_{X_1}
        \norm{\sum_{k=0}^N\tau_{y_k}g}_{X_2},
    \]
    where recall that $(E_t f)(s)=f^*(ts)$ denotes the dilation operator (see \eqref{Dilation}  above). By  \eqref{eq:tech_2} and \eqref{eq:tech_4}  the term on the right hand is equal to
    \[
        \norm{\widehat{K}} \norm{E_{N+1}f^*}_{X_1^*}
        \norm{E_{N+1}g^*}_{X_2^*}.
    \]
    Therefore, by \eqref{Dilation_norm}, 
    \[
        0<\norm{B_{\widehat{K}}(f,g)}_{X}\leq \norm{\widehat{K}}
        h_{X}\brkt{\frac{1}{N+1}} h_{X_1}(N+1) h_X{X_2}(N+1)
        \norm{f}_{X_1}\norm{g}_{X_2}.
    \]
 Hence, since $f,g$ are fixed, this implies that there exists a constant $c>0$ such that for any $N$
    \[
         h_{X}\brkt{\frac{1}{N+1}} h_{X_1}(N+1) h_{X_2}(N+1)>c,
    \]
    which,  by \eqref{eq:boydsup} and \eqref{eq:boydinf}, yields
    that
    \[
        \underline{\alpha}_{X}\leq
        \overline{\alpha}_{X_1}+\overline{\alpha}_{X_2}.
    \]
\end{proof}

\begin{rem} 
Observe that in the previous proof, we have used the convexity assumptions only for being able to apply Theorem \ref{thm:Main}, in order to ensure the existence of a multiplier $\mul$, which Fourier transform is a compactly supported integrable function. Hence, we could have dropped the convexity conditions if we have imposed this last condition on $\mul$ instead. 
\end{rem}

As an application of the previous theorem we can obtain an extended version of L. Grafakos and J. Soria's result \cite{MR2595656}*{Theorem 1} to RI QBFSs.
\begin{cor}  Let $G$ be a non-compact LCA group and let $X_1,X_2, X_2$ RI QBFS on $G$. If there exists $K\in L^1(G^2)\setminus \{0\}$, $K\geq 0$, such that
$\widehat{K}\in\BM{X_{1},X_{2}}{X}{G}$, then
\[
    \underline{\alpha}_{X}\leq  \overline{\alpha}_{X_1}+\overline{\alpha}_{X_1}.
\]
\end{cor}
\begin{proof}  Since $K\geq 0$, the boundedness of $B_K$ is equivalent to the boundedness of the operator 
\[
	P_K(f,g)(v)=\iint_{G^2} K(u_1,u_2) \abs{f(v-u_1)} \abs{g(v-u_2)}\, \dd u_1\dd u_2,
\]
and in particular, for any compact set in $\mathcal{K}\subset G^2$ such that $K$ is not zero on $\mathcal{K}$, $P_{K\chi_{\mathcal{K}}}$ defines a bounded operator from $X_1\times X_2\to X$. Then $K\chi_{\mathcal{K}}\in L^{1}_c(G^2)$ and $\widehat{K\chi_{\mathcal K}}\in \BM{X_{1},X_{2}}{X}{G}$. Thus the previous remark and Theorem \ref{thm:necessity} yield the result.
\end{proof}

If we particularise Theorem \ref{thm:necessity} to the case of classical Lorentz-spaces, we obtain, is the following extension of F. Villarroya's result \cite[Proposition 3.1
]{MR2471164} to arbitrary non compact LCA groups. 
\begin{cor}  Let $G$ be a non-compact LCA group and let $1< p_1,p_2<\infty$, $1\leq q_1,q_2\leq \infty$. If there exists $\mul\in \BM{L^{p_1,q_1},L^{p_2,q_2}}{L^{p,q}}{G}$, $\mul\neq 0$, then
\[
	\frac{1}{p}\leq \frac{1}{p_1}+\frac{1}{p_2}.
\]
\end{cor}

%\bibliographystyle{plainnat}
%\bibliographystyle{abbrv}
%\bibliographystyle{amsalpha}
%\bibliographystyle{alpha}
%\bibliographystyle{alpha}
%\bibliography{multilinear}
%\begin{thebibliography}
%\end{thebibliography}

\end{document}